\documentclass[12pt]{article}

 \usepackage{graphics}
 \usepackage{graphicx}
 \usepackage{epstopdf}
 \usepackage{epsfig}
 \usepackage[all]{xy}
 \usepackage{cite}

 \usepackage{amssymb}
 \usepackage{amsthm,amsmath}

 \usepackage{latexsym, epsfig, psfrag,eepic,colordvi,bm}
 \usepackage{graphics,color}
 \usepackage{amsmath,amsfonts,amssymb,amscd}
 \usepackage[all]{xy}
 \usepackage{epstopdf}
 \usepackage{mathrsfs}

 \usepackage{amsthm,amsmath,amssymb}

 \usepackage{graphicx}
\usepackage{amsthm,amsmath,amssymb}

\usepackage{graphicx,epsfig}

\usepackage[colorlinks=true,citecolor=black,linkcolor=black,urlcolor=blue]{hyperref}
\usepackage{arydshln}

\newcommand{\genstirlingI}[3]{%
	\genfrac{[}{]}{0pt}{#1}{#2}{#3}%
}

\newcommand{\stirlingI}[2]{\genstirlingI{}{#1}{#2}}

\theoremstyle{plain}
\newtheorem{theorem}{Theorem}[section]
\newtheorem{lemma}[theorem]{Lemma}
\newtheorem{corollary}[theorem]{Corollary}
\newtheorem{proposition}[theorem]{Proposition}

\theoremstyle{definition}

\theoremstyle{remark}

\date{}
\title{\bf Explicit formulas for a family of hypermaps beyond the one-face case}

\author{Zi-Wei Bai, Ricky X. F. Chen\\
	\small School of Mathematics, Hefei University of Technology\\[-0.8ex]
	\small Hefei, Anhui 230601, P.~R.~China\\[-0.8ex]
	\small\tt xiaofengchen@hfut.edu.cn
}

\begin{document}
\maketitle
\noindent{\bf Abstract.}
Enumeration of hypermaps is widely studied in many fields.
In particular, enumerating hypermaps with a fixed edge-type according to the number of faces and genus is one topic of great interest.
However, it is challenging and explicit results mainly exist for hypermaps having one face, especially for the 
edge-type corresponding to maps.
The first systematic study of one-face hypermaps with any fixed edge-type is the work of Jackson (Trans.~Amer.~Math.~Soc.~299, 785--801, 1987)
using group characters.
In 2011, Stanley obtained the generating polynomial of one-face hypermaps of any fixed edge-type expressed in terms of the backward shift operator.
There are also enormous amount of work on enumerating one-face hypermaps of specific edge-types.
The enumeration of hypermaps with more faces is generally much harder.
In this paper, we make some progress in that regard, and obtain the generating polynomials and properties for a family of typical two-face hypermaps with almost any fixed edge-type.

\vskip 10pt

\noindent{\bf Keywords:}  Hypermap, Permutation product, Genus distribution, Cycle distribution,
Group character, Log-concavity

\noindent\small Mathematics Subject Classifications 2020: 05E10, 05A15, 20B30

\section{Introduction}

A map is a $2$-cell embedding of a connected graph in an orientable surface~\cite{lan-zvon}.
Hypermaps are generalization of maps.
Equivalently, hypermaps can be viewed as bipartite (or bi-colored) maps~\cite{bipartite}.
A bipartite map is a map where there exists a way of coloring the vertices such that two adjacent vertices
have different colors and at most two colors in total are used.
Bipartite maps are also known as Grothendieck's dessins d'enfants in studying Riemann surfaces~\cite{zap,lan-zvon}.
These objects lie at the crossroad of a number of research fields, e.g.,
topology, mathematical physics and representation theory.

A map in an orientable surface can be encoded into a rotation system on the underlying graph,
i.e., the graph together with a cyclic order of edge-ends incident to every vertex of the graph (see, e.g., Edmonds~\cite{edmonds}).
An edge-end is sometimes called a dart or half-edge.
See an example of a bipartite map (with edge labels) in Figure~\ref{fig:exam1}.
A rooted map (thus hypermap) is a map where one edge-end is distinguished and called the root.
In this paper, we only consider rooted maps and hypermaps.
Moreover, besides connected maps, disconnected (hyper)maps, i.e., with the underlying graph disconnected, are also included. 

	\begin{figure}\label{fig:exam1}
	\centering
	\includegraphics[scale=0.5]{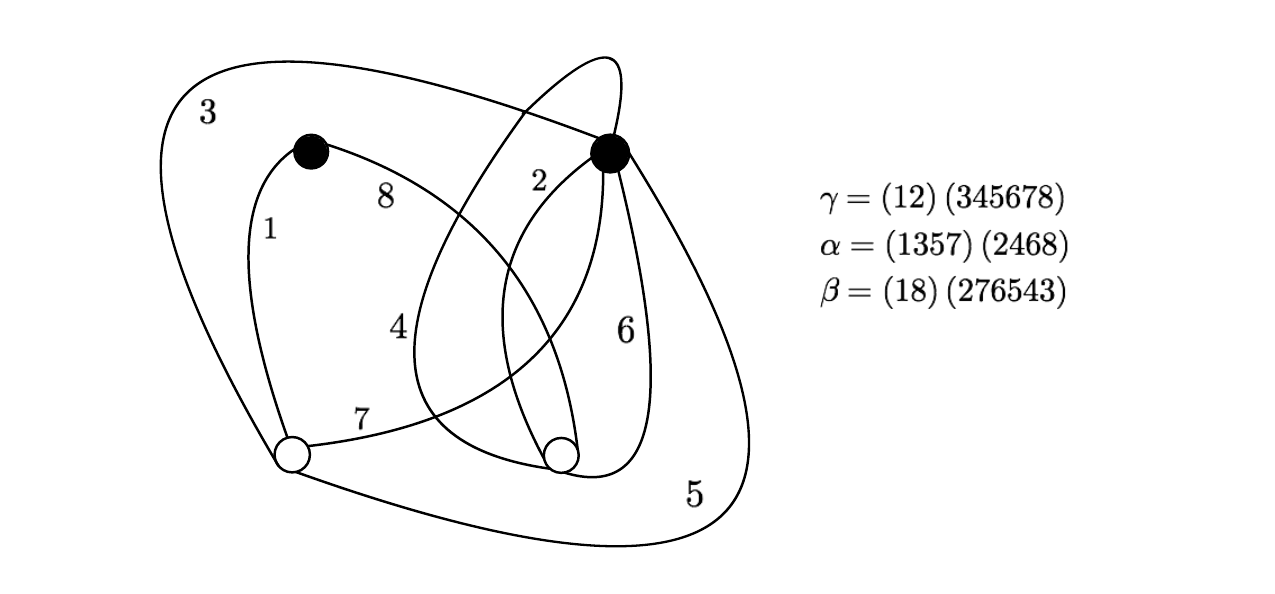}
	\caption{A bipartite map.} 
\end{figure}

Let $[n] = \{1, 2,\ldots , n\}$, and let $\mathfrak{S}_n$ denote the group of permutations on $[n]$. The number of disjoint
cycles of $\pi \in \mathfrak{S}_n$ is denoted by $\kappa(\pi)$, and the set (with repetition allowed) consisting of the
lengths of these disjoint cycles is called the cycle-type of $\pi$.
Let $\rho(\pi)$ denote the cycle-type of $\pi$.
We write $\rho(\pi)$ as a partition of $n$, i.e., a nonincreasing positive integer sequence $\lambda=(\lambda_1, \lambda_2, \ldots, \lambda_k)$
such that $\lambda_1+\cdots + \lambda_k = n$.
If in $\lambda$, there are $m_i$ of $i$'s, we also write $\lambda$ as $[1^{m_1}, 2^{m_2}, \ldots, n^{m_n}]$,
with $i^{m_i}$ sometimes omitted when $m_i=0$, and with $i^1$ simply written as $i$.
A long cycle or $n$-cycle or cyclic permutation on $[n]$ is a permutation of cycle-type $(n)$.
A permutation on the set $[2n]$ of cycle-type $[2^n]$ is called a fixed-point free involution.

Algebraically, the study of hypermaps is studying products of permutations in symmetric groups.
Specifically, a triple $(\alpha, \beta, \gamma)$ of permutations on the set $[n]$ such that $\gamma=\alpha \beta$ (compose from right to left)
determines a possibly disconnected hypermap,
where the cycles of $\gamma$ encode the faces, the cycles of $\alpha$ encode the ``edges" (hyperedges),
and the cycles of $\beta$ encode the ``vertices" of the hypermap.
Equivalently, in terms of bipartite maps, the cycles of $\gamma$ encode the faces, the cycles of $\alpha$ encode the white vertices,
and the cycles of $\beta$ encode the black vertices,
and an edge connects the same label in $\alpha$ and $\beta$.
Taking the bipartite map in Figure~\ref{fig:exam1} as an example, the black vertices give the permutation (in cycle notation) $\beta= (18)(276543)$,
and the white vertices give the permutation $\alpha= (1357)(2468)$.
A face is a region bounded by a cyclic sequence of edges $(\ldots, i, i', \ldots)$,
where $i'$ is the edge obtained by starting with the edge $i$, going counterclockwisely to the
neighboring edge $e$ around a black vertex, and then going counterclockwisely to the neighboring edge
around a white vertex.
For example, in Figure~\ref{fig:exam1}, if $i=2$, then its counterclockwise neighbor around a black vertex is $7$ whose
counterclockwise neighbor around a white vertex is $1$. Thus, $i'=1$.
Consequently, we see that a face corresponds to a cycle in $\gamma$.
If $\gamma$ has only one cycle, then it is not hard to argue that the corresponding
map is always connected.

According to the Euler characteristic formula, the genus $g$ of a map satisfies:
\begin{align*}
	v-e+f=2-2g,
\end{align*}
where $v, e, f$ are respectively the numbers of vertices, edges and faces of the map.
This translates to the following relation:
\begin{align}\label{eq:genus}
	\kappa(\alpha)+\kappa(\beta)-n+\kappa(\gamma)=2-2g.
\end{align}

Enumerating certain hypermaps, e.g., enumerating rooted hypermaps with a fixed edge-type according to the number of faces and genus, is thus known to be equivalent to counting permutation products satisfying certain conditions.
For instance, the number of rooted hypermaps with a fixed edge-type (i.e., the cycle-type of $\alpha$) and having one face as well as a prescribed genus
equals the number of products of a long cycle (i.e., $\gamma$) and a permutation of a fixed cycle-type (i.e., $\alpha$)
which give a permutation (i.e., $\beta^{-1}=\gamma \alpha$) having a prescribed number of cycles subject to eq.~\eqref{eq:genus}.
Note that hypermaps with the edge-type corresponding to fixed-point free involutions are ordinary maps.
Although as easy as it is to formulate the problem, it is very challenging to obtain explicit formulas.
To the best of our knowledge, results mainly exist for one-face hypermaps.
Even for the latter restricted case, explicit and relatively simple formulas are only known for one-face
hypermaps with a few edge-types.
For example, one-face hypermaps with the edge-type $(n)$ are counted by the Zagier-Stanley formula~\cite{zag,Stanely-two,cms,chen-reidys};
One-face hypermaps with the edge-type $[2^n]$ (i.e., maps) are counted by a variety of formulas and recurrences,
e.g., the Walsh-Lehman formula~\cite{walsh1}, the Harer-Zagier formula and recurrence~\cite{harer-zagier}, and Chapuy's recursion~\cite{chapuy11};
One-face hypermaps with the edge-type $[1, n-1]$ can be counted by a result of Bocarra~\cite{boccara}.
See also~\cite{GN05,ber-morales,fv,cff,chr-products,chr-versatile,IJ2,Goupil,Jackson-combinatorial,MV,stanely-Factorization,sch-vass} and the references therein.

The first systematic study of enumerating one-face hypermaps of any fixed edge-type according to genus was the work
of Jackson~\cite{Jackson}, where he obtained a generating function of certain form.
Later, in Goupil and Schaeffer~\cite[Theorem~4.1]{Goupil}, an explicit formula involving multiple summations was obtained.
Stanley~\cite{Stanely-two} studied the following cycle (thus genus) distribution polynomial for one-face hypermaps:
\begin{align}
	P_{\lambda}(q)= \sum_{\omega \in \mathfrak{S}_n, \, \rho(\omega)= \lambda} q^{\kappa(\alpha \omega)},
\end{align}
where $\alpha$ is a fixed long cycle on $[n]$.
An expression for $P_{\lambda}(q)$ was given in terms of the backward shift operator,
and it was shown that $P_{\lambda}(q)$ has purely imaginary zeros.
In addition, for $\lambda=(n)$, the coefficients of the polynomial were
explcitly determined, i.e., the Zagier-Stanley formula.
See later combinatorial proofs and refinements in~\cite{cms,chr-products,chr-versatile,fv}.
However, for a general $\lambda$, it is still not easy to extract coefficients.
The second author~\cite{chen} recently obtained a new explicit formula for one-face hypermaps with any fixed edge-type from which
the Zagier-Stanley formula, the Harer-Zagier formula and the Jackson formula~\cite{Jackson-combinatorial,sch-vass} etc.~can be easily derived.
Moreover, a dimension reduction recurrence was also obtained.

As for hypermaps beyond the one-face case, known explicit formulas are rare.
It is our goal here to contribute some.
The reason why the one-face case has been so fruitful is not because it is the only interesting case but there are many nice properties.
For example, for the one-face case, there is only one face-type (i.e., the cycle-type of $\gamma$);
in terms of group characters, only hook-shape Young diagrams matter due to the involvement of the conjugacy class of cycle-type $(n)$.
However, when there are more than one face, these properties will vanish.
Maybe the next simpler case other than the one-face case is two-face hypermaps.
The face-types of two-face hypermaps are of the form $[p, n-p]$.
Although not discussed in the literature, we believe that the enumeration of two-face hypermaps
of face-type $[1,n-1]$ may be derived from the one-face case.
Therefore, the starting interesting case of two-face hypermaps is the ones of face-type $[2,n-2]$
which are the objects of our study here.
Actually, the one in Figure~\ref{fig:exam1} corresponds to a hypermap of face-type $[2, n-2]$. 

Our main result is an unexpected simple expression for the polynomial
$$
P_{[2,n-2], \beta} (x)=\frac{ n (n-1)}{|\mathcal{C}_{\beta}|} \sum_{\omega \in \mathfrak{S}_n, \, \rho(\omega)= \beta}   x^{\kappa(\alpha \omega)},
$$
where $\alpha$ is a fixed permutation of cycle-type $[2,n-2]$, $|\mathcal{C}_{\beta}|$
is the number of permutations of cycle-type $\beta$, and $\beta$ could be any partition of $n$
except a few cases. Based on this, we also derive a number of interesting results.
For instance, we show that any $P_{[2,n-2], \beta} (x)$ is a linear combination
of $n$ fixed simpler polynomials and has only imaginary roots.
The latter implies the log-concavity of the coefficients.

The paper is organized as follows.
In Section~\ref{sec2}, we review the main tools that will be
used to attack the problem.
In Section~\ref{sec3}, we compute $P_{[2,n-2], \beta} (x)$ and present the decomposition result.
In Section~\ref{sec4}, we prove the imaginary zero property and log-concavity.

\section{Preliminaries} \label{sec2}

It is well known that a conjugacy class of $\mathfrak{S}_n$ contains
permutations of the same cycle-type.
So, the conjugacy classes can be indexed by partitions of $n$.
Let $\mathcal{C}_{\lambda}$ denote the one indexed by $\lambda$.
The number of elements contained in $\mathcal{C}_{\lambda}$ is well known to be
$$
|\mathcal{C}_{\lambda}|=\frac{n!}{z_{\lambda}}, \quad \mbox{where $z_{\lambda}= \prod_{i=1}^n i^{m_i} m_i!$,}
$$
and permutations on $[n]$ having exactly $k$ cycles are counted by the signless Stirling
number of the first kind $\stirlingI{n}{k}$.

\begin{lemma} \label{lem:stirling}
	 The signless Stirling
	numbers of the first kind satisfy
	\begin{align}
\sum_{k=1}^n (-1)^{n-k}\stirlingI{n}{k} x^k = x(x-1)\cdots (x-n+1).
	\end{align}
\end{lemma}

We write the character associated to the irreducible representation indexed by $\lambda$
as $\chi^{\lambda}$ and the dimension of the irreducible representation as $f^{\lambda}$.
The readers are invited to consult Stanley~\cite{stan-ec2}, and Sagan~\cite{sagan} for the character theory of symmetric groups.

\begin{lemma}[The hook length formula]
	\begin{align}
		f^{\lambda} = \frac{n!}{ \prod_{u \in \lambda} h_u},
	\end{align}
where for $u=(i,j)\in \lambda$, i.e., the $(i,j)$ cell in the Young diagram of $\lambda$, $h(u)$ is the hook length of
the cell $u$.
\end{lemma}

 Suppose $C$ is a conjugacy class of $\mathfrak{S}_n$ indexed by $\alpha$ and $\pi \in C$.
We view $\chi^{\lambda}(\pi)$, $\chi^{\lambda}(\alpha), \, \chi^{\lambda}(C)$ as the same, and we trust the context to prevent confusion.
Let
\begin{align*}
	\mathfrak{m}_{\lambda,m}= \prod_{u\in \lambda} \frac{m+c(u)}{h(u)}, \qquad 	\mathfrak{c}_{\lambda,m} &= \sum_{d=0}^m (-1)^d {m \choose d} \mathfrak{m}_{\lambda,m-d}.
\end{align*}
where for $u=(i,j)\in \lambda$, $c(u)=j-i$.
The following theorem was proven in~\cite{chr-hur}.

\begin{theorem}[Chen~\cite{chr-hur}] \label{thm:main-recur11}
	Let $\xi_{n,m}(C_1,\ldots, C_t)$ be the number of tuples $( \sigma_1,\sigma_2,\ldots ,\sigma_t)$
	such that the permutation $\pi=  \sigma_1\sigma_2\cdots \sigma_t$ has $m$ cycles, where $\sigma_i $
	belongs to a conjugacy class ${C}_i$. Then we have
	\begin{align}\label{eq:main-recur11}
		\xi_{n,m}(C_1,\ldots, C_t)
		= \bigg(\prod_{i=1}^{t}|C_i| \bigg)\sum_{k = 0}^{n-m} (-1)^k \stirlingI{m+k}{m}  \frac{1}{(m+k)!}  W_{n,m+k}(C_1,\ldots, C_t),
	\end{align}
	where
	\begin{align}\label{eq:W}
		W_{n,m+k}(C_1,\ldots, C_t)=  \sum_{\lambda \vdash n} \frac{\mathfrak{c}_{\lambda,m+k} } {(f^{\lambda})^{t-1}}	  \prod_{i=1}^t \chi^{\lambda}(C_i).
	\end{align}
	
\end{theorem}

We remark that when $m=n$, eq.~\eqref{eq:main-recur11}
reduces to the famous Frobenius identity. See Chen~\cite{chr-hur} for discussion.
Theorem~\ref{thm:main-recur11} has been used to study one-face hypermaps,
and a plethora of results have been obtained in a unified way~\cite{chen}.
In the following, we will further explore Theorem~\ref{thm:main-recur11}
and make some plausible progress beyond the one-face case.

\section{Hypermaps of face-type $[2, n-2]$} \label{sec3}

In the notation of Theroem~\ref{thm:main-recur11},
we will be interested in the case where $t=2$, $C_1$ corresponds to $[2, n-2]$ and $C_2$ corresponds
to $\beta=(\beta_1, \ldots, \beta_d) \vdash n$.
The corresponding $W$-number in this case is as follows:

	\begin{align}\label{eq:WW}
	W_{n,m+k}=  \sum_{\lambda \vdash n} \frac{\mathfrak{c}_{\lambda,m+k} }{	 f^{\lambda} }  \chi^{\lambda}([2, n-2]) \, \chi^{\lambda}(\beta) .
\end{align}

In the following, we will distinguish two cases: (i) any $\beta$ with its minimum size of a part $\min(\beta)$ no less than three, (ii)
$\beta=[2, n-2]$ and $\beta=[1,n-1]$, i.e., special cases where the minimum size of a part of $\beta$ is exactly two or one.

\subsection{$\min(\beta) \geq 3$}

In this part, we assume
\begin{align*}
	\beta=(\beta_1, \beta_2, \ldots, \beta_d)= [l^{n_l}, \ldots, q^{n_q}] \vdash n,
\end{align*}
where $l$ and $q$ are respectively the mininum size and maximum size of a part in $\beta$.
We first compute the two characters involved in $W_{n,m+k}$ using the Murnaghan-Nakayama rule (see, e.g., Stanley~\cite{stan-ec2}, Sagan~\cite{sagan}).

\begin{lemma} \label{lem:2}
For $\lambda \vdash n \geq 6$, we have
\begin{align}
  \chi^{\lambda}([2,n-2])= \begin{cases}
  (-1)^j ,  & \lambda =[j,n-j], \, j=0,1,2, \\
  (-1)^{j+1} , & \lambda =[1^j,3,n-j-3], \mbox{  $0\leq j \leq n-6$},\\
  (-1)^{j+1} , & \lambda =[1^j,2^2,n-j-4], \mbox{ $0\leq j \leq n-6$},\\
  (-1)^n  , & \lambda =[1^{n-4},2^2],\\
  (-1)^{n-1} , & \lambda =[1^{n-2},2],\\
  (-1)^n , & \lambda =[1^n],\\
  0, &others.
\end{cases}
\end{align}
\end{lemma}

\begin{proof}
	According to the Murnaghan-Nakayama rule, the $\lambda$ which may return a nonzero value
	correspond to the Young diagrams in Figure \ref{fg1}.
	\begin{figure}\label{fg1}
	\centering
	\includegraphics[scale=0.4]{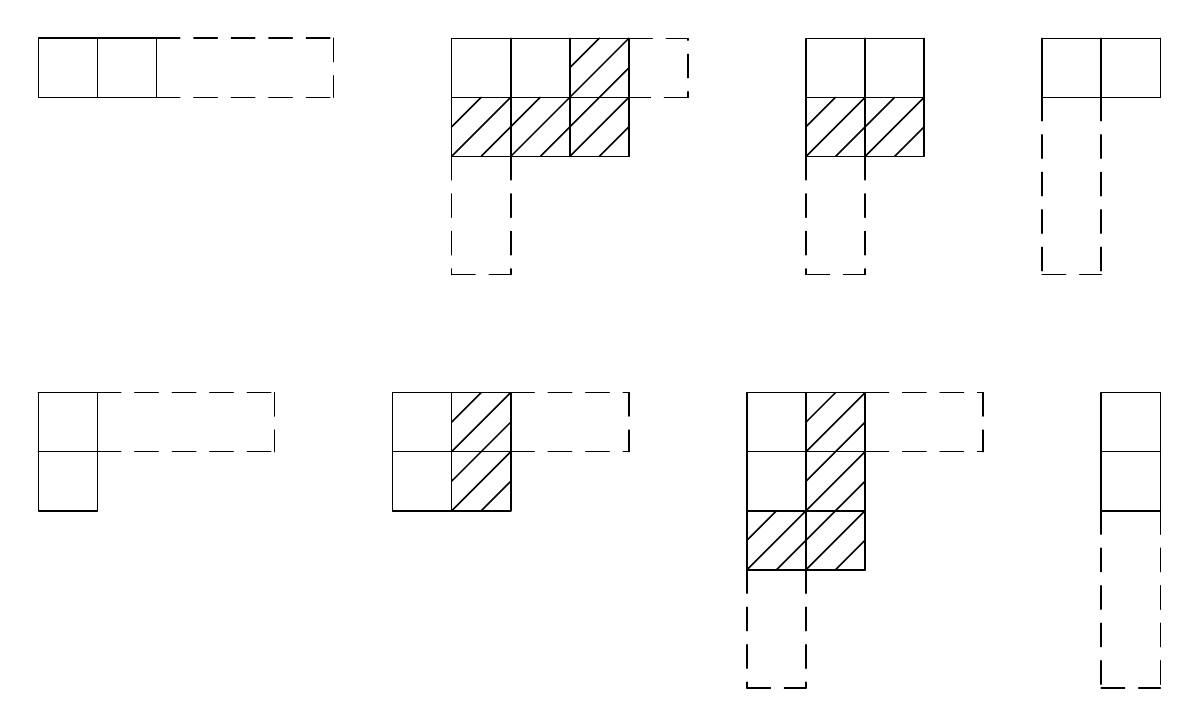}
	\caption{The $\lambda$'s for that $\chi^{\lambda}([2,n-2])$ may not be zero.} \label{111111}
\end{figure}
	The rest is straightforward and the lemma follows.
\end{proof}

The calculation of $\chi^{\lambda}(\beta)$ is crutial and generally hard.
However, for our purpose here, it suffices to compute for $\lambda$ which gives a non-zero value in Lemma~\ref{lem:2}.

\begin{lemma}\label{lem:beta}
There holds	
	\begin{align}\label{eq:beta-simple}
		\chi^{\lambda}(\beta)=	\begin{cases}
			(-1)^j, & \lambda =[j,n-j], \, j=0,1,\\
			0, & \lambda =[2,n-2], \mbox{ and $\lambda=[1^{n-4}, 2^2]$},\\
			(-1)^{n-\ell(\beta)-1},  & \lambda =[1^{n-2},2],\\
			(-1)^{n-\ell(\beta)},  & \lambda =[1^n].
		\end{cases}
	\end{align}
	If $\lambda =[1^j,3,n-j-3]$ or $[1^j,2^2,n-j-4]$ for $0\leq j \leq n-6$, we have
	\begin{align}
		\chi^{\lambda}(\beta) &= 	\sum_{{j_l}=0}^{n_{l}-1}\sum_{{j_{l+1}}=0}^{n_{l}}\cdots\sum_{{j_q}=0}^{n_q}(-1)^{j-\sum_{i\geq l} j_{i}}{{n_{l}-1}\choose{j_{l}}}
		{n_{l+1}\choose j_{l+1}}\cdots{{n_q}\choose j_q} \nonumber\\
		& \qquad \qquad \qquad \times \big\{\delta_{j+3-l,\sum_{i\geq l}ij_i}-\delta_{j+3,\sum_{i\geq l}ij_i} \big\} ,
	\end{align}
	where $\delta_{x, y}=1$ if $x=y$, and $0$ otherwise.
\end{lemma}

\proof The formulas in eq.~\eqref{eq:beta-simple} should be easy to verify and we omit their proofs.
We next consider the latter statement.
We only prove the case $\lambda=[1^j, 3, n-j-3]$,
and the other is analogous.
For the considered case here, there are only two cases which may not vanish according to the Murnaghan-Nakayama rule:
\begin{itemize}
\item[(i)] The rim hook (or border strip) containing the northwest corner cell (i.e., the cell $(1,1)$) consists of the cells
$(1,1), (1,2), \ldots, (1,l)$; See an illustration in Figure~\ref{fig:chi-beta} left for $l=3$;
\item[(ii)] The rim hook containing the northwest corner cell consists of the cells $(1,1)$,$(2,1),\ldots, (l-1,1)$ and
$(1,2)$; See Figure \ref{fig:chi-beta} right for $l=3$.
\end{itemize}
	\begin{figure}\label{fig:chi-beta}
	\centering
	\includegraphics[scale=0.5]{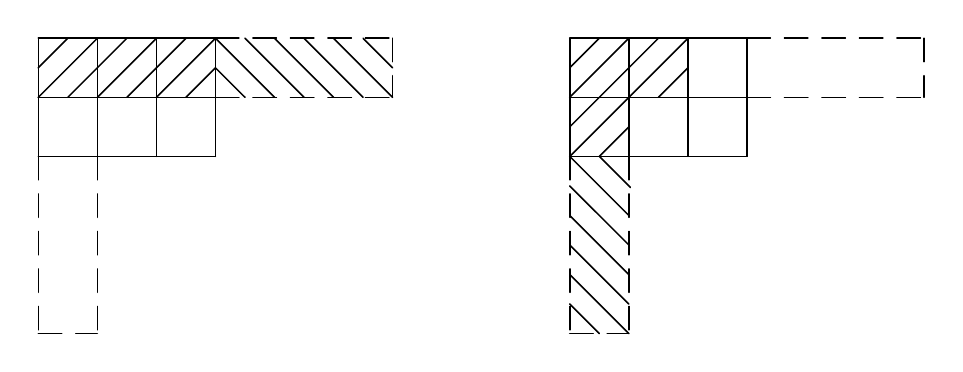}
	\caption{The two cases may not vanish.} 
\end{figure}
For (i), we first need to remove $n_q$ rim hooks of size $q$ either from the horizontal portion with backslash or
from the hook-shape portion with plain cells. But, we have to record which $j_q$ rim hooks are from the latter.
There are obviously ${n_q \choose j_q}$ possibilities and each contributes a factor $(-1)^{(q-1) j_q}$ (from the heights of the rim hooks).
Next, we need to remove $n_i$ rim hooks of size $i$ from $i=q-1$ to $i=l+1$ one by one analogously, $j_i$ of which
are from the plain portion,
which contributes a factor $(-1)^{\sum_{i=l+1}^{q-1}( i-1) j_i}$.
Finally, we have to remove $n_{l}-1$ rim hooks of size $l$, $j_l$ of which are from the plain portion.
Note that this is possible only if $j+3= \sum_{i=l}^{q} i j_i$.
When it is possible, the contributed factor from the last steps is $(-1)^{l j_l-2-j_l}$.
Note that when $j+3= \sum_{i=l}^{q} i j_i$,
$$
(-1)^{(q-1) j_q} (-1)^{\sum_{i=l+1}^{q-1}( i-1) j_i} (-1)^{l j_l-2-j_l} = (-1)^{j+3-2-\sum_i j_i}.
$$
Thus, the contribution to $\chi^{\lambda}(\beta)$ from (i) can be expressed as
$$
\sum_{{j_l}=0}^{n_{l}-1}\sum_{{j_{l+1}}=0}^{n_{l}}\cdots\sum_{{j_q}=0}^{n_q}(-1)^{j-\sum_{i\geq l} j_{i}}{{n_{l}-1}\choose{j_{l}}}
{n_{l+1}\choose j_{l+1}}\cdots{{n_q}\choose j_q}
 (-\delta_{j+3, \, \sum_{i\geq l}ij_i} ).
$$
The computation of the contribution from (ii) is similar, and the proof follows.
\qed

Although we have omitted lots of details, it worth pointing out that $\chi^{\lambda}(\beta)$
having the same value for $\lambda=[1^j,2^2,n-j-4]$ and $\lambda=[1^j,3,n-j-3]$ is a little surprise.
Next, we compute the factor $\frac{\mathfrak{c}_{\lambda,m+k} }{	 f^{\lambda} }$ in eq.~\eqref{eq:WW}, only for these $\lambda$ which
return a non-zero value in Lemma~\ref{lem:beta}.
First of all, we have
\begin{align}\label{eq:cf}
\frac{\mathfrak{c}_{\lambda,m} }{	 f^{\lambda} }=\frac{ \sum_{d=0}^m (-1)^d {m \choose d}  \prod_{u\in \lambda} \frac{m-d+c(u)}{h(u)} }{\frac{n!}{ \prod_{u \in \lambda } h_u}} =\frac{1}{n!}  \sum_{d=0}^m (-1)^d {m \choose d}  \prod_{u\in \lambda} \{m-d+c(u)\} .
\end{align}

\begin{lemma} \label{lem:3}
Suppose $\lambda=[1^j,3,n-j-3]$ for $0\leq j\leq n-6$. Then, we have
\[\frac{\mathfrak{c}_{[1^j,3,n-j-3],m}}{f^{[1^j,3,n-j-3]}}=\frac{m(m-1)}{n(n-1)}{{n-j-2}\choose{n-m}}+\frac{2m}{n(n-1)}{{n-j-3}\choose{n-m-1}}.\]
\end{lemma}

\begin{proof}
	According to eq.~\eqref{eq:cf}, we obtain
\begin{equation}
   \begin{split}
   \frac{\mathfrak{c}_{[1^j,3,n-j-3],m}}{f^{[1^j,3,n-j-3]}}&=\frac{1}{n!}\sum_{d=0} ^m(-1)^d{m\choose d}(m-d)(m-d+1)\cdots(m-d+n-j-3-1)\\
   &\quad \times (m-d-1)(m-d)
   (m-d+1)(m-d-2)\cdots(m-d-j-1)\\
   &=\frac{1}{n!}\sum_{d=0}^m (-1)^d{m\choose d}\frac{(m-d+n-j-4)!}{(m-d-j-2)!}(m-d)(m-d-1+2)\\
   &=\frac{m(m-1)}{n(n-1)}\sum_{d=0}^m[x^d](1-x)^{m-2}[x^{m-d-j-2}]\frac{1}{(1-x)^{n-1}}\\
   &\quad +\frac{2m}{n(n-1)}\sum_{d=0}^m[x^d](1-x)^{m-1}[x^{m-d-j-2}]\frac{1}{(1-x)^{n-1}}\\
   &=\frac{m(m-1)}{n(n-1)}[x^{m-j-2}]\frac{1}{(1-x)^{n-m+1}}+\frac{2m}{n(n-1)}[x^{m-j-2}]\frac{1}{(1-x)^{n-m}}\\
   &=\frac{m(m-1)}{n(n-1)}{{n-j-2}\choose{n-m}}+\frac{2m}{n(n-1)}{{n-j-3}\choose{n-m-1}},
  \end{split}\nonumber
\end{equation}
and the lemma follows.
\end{proof}

Analogously, we obtain the formulas for other cases which are presented in the forthcoming lemma.
\begin{lemma}\label{lem:cf}
	The following is true:
\begin{align}
		\frac{\mathfrak{c}_{\lambda,m}}{f^{\lambda}}=
			\begin{cases}
			{{n-1}\choose{m-1}}, & \lambda =[n], \\
			{0\choose{m-n}},  & \lambda =[1^{n}], \\
			{1 \choose {n-m}}, & \lambda=[1^{n-2},2],\\
			{{n-2}\choose{n-m}}, & \lambda=[1, n-1], \\
			\frac{m(m-1)}{n(n-1)}{{n-j-3}\choose{n-m}}, & \lambda=[1^j,2^2,n-j-4], \quad 0\leq j\leq n-6.
		\end{cases}
	\end{align}
\end{lemma}

In view of Lemma~\ref{lem:2} and Lemma~\ref{lem:beta}, we can break the quantity $W$ into the following form:
\begin{align*}
	{	W_{n,m+k}} &=   \frac{\mathfrak{c}_{\lambda,m+k} }{	 f^{\lambda} }  \chi^{\lambda}([2, n-2]) \chi^{\lambda}(\beta) {\bigg |}_{\lambda=[n]} +\frac{\mathfrak{c}_{\lambda,m+k} }{	 f^{\lambda} }  \chi^{\lambda}([2, n-2]) \chi^{\lambda}(\beta) {\bigg |}_{\lambda=[1^{n}]}\\
 & \quad + \frac{\mathfrak{c}_{\lambda,m+k} }{	 f^{\lambda} }  \chi^{\lambda}([2, n-2]) \chi^{\lambda}(\beta) {\bigg |}_{\lambda=[1^{n-2},2]}+\frac{\mathfrak{c}_{\lambda,m+k} }{	 f^{\lambda} }  \chi^{\lambda}([2, n-2]) \chi^{\lambda}(\beta) {\bigg |}_{\lambda=[1,n-1]} \\
 & \quad + \sum_{\lambda=[1^j,2^2,n-j-4], \atop 0\leq j \leq n-6} \frac{\mathfrak{c}_{\lambda,m+k} }{	 f^{\lambda} }  \chi^{\lambda}([2, n-2]) \chi^{\lambda}(\beta)+\sum_{\lambda=[1^j,3,n-j-3], \atop 0\leq j \leq n-6} \frac{\mathfrak{c}_{\lambda,m+k} }{	 f^{\lambda} }  \chi^{\lambda}([2, n-2]) \chi^{\lambda}(\beta).
\end{align*}

We proceed to compute the terms in $W_{n, m+k}$ one by one, and we also compute the generating function of each term.

\begin{proposition}\label{n}
	We have
	\begin{align}
		A_1= \frac{\mathfrak{c}_{\lambda,m+k} }{f^{\lambda} }  \chi^{\lambda}([2, n-2]) \chi^{\lambda}(\beta) {\bigg |}_{\lambda=[n]} &={{n-1}\choose{n-m-k}} ,\\
	 \sum_{m \geq 1} \sum_{k=0}^{n-m} (-1)^k \stirlingI{m+k}{m} \frac{n(n-1)}{(m+k)!} A_1 x^m
 &=n(n-1) {x+n-1 \choose n}.
\end{align}
\end{proposition}

\proof
The quantity $A_1$ immediately follows from Lemma~\ref{lem:2}, Lemma~\ref{lem:beta} and Lemma~\ref{lem:cf}.
We proceed to compute
\begin{align*}
	& \quad \sum_{m \geq 1} \sum_{k=0}^{n-m} (-1)^k \stirlingI{m+k}{m} \frac{n(n-1)}{(m+k)!} A_1 x^m\\
	 &=\sum_{m \geq 1}\sum_{k=0}^{n-m} (-1)^k \stirlingI{m+k}{m} \frac{n(n-1)}{(m+k)!} {{n-1}\choose{n-m-k}} x^m \\
	&=\sum_{m \geq 1} \sum_{k\geq1} (-1)^{k-m} \stirlingI{k}{m} \frac{n(n-1)}{k!} {{n-1}\choose{n-k}} x^m .
	\end{align*}
The range of $k$ in the last formula is from $1$ to $n$. However, since when $k>n$, ${n-1 \choose n-k}=0$, it is safe
to write $k\geq 1$. Next, by exchanging the order of the two summations and applying Lemma~\ref{lem:stirling}, the last formula equals 
\begin{align*}
\sum_{k\geq1}{x\choose k}n(n-1){{n-1}\choose{n-k}}
	&=n(n-1) {x+n-1 \choose n},
\end{align*}
completing the proof. \qed

Applying the similar strategy, we obtain the following propositions in order.
\begin{proposition}\label{1n}
	\begin{align}
		A_2= \frac{\mathfrak{c}_{\lambda,m+k} }{	 f^{\lambda} }  \chi^{\lambda}([2, n-2]) \chi^{\lambda}(\beta) {\bigg |}_{\lambda=[1^n]} &={0\choose{n-m-k}}(-1)^{\ell(\beta)} ,\\
		\sum_{m \geq 1}\sum_{k=0}^{n-m} (-1)^k \stirlingI{m+k}{m} \frac{n(n-1)}{(m+k)!}A_2 x^m
&=n(n-1){x\choose n}(-1)^{\ell(\beta)} .
	\end{align}
\end{proposition}

\begin{proposition}\label{1n-2,2}
	\begin{align}
		A_3= \frac{\mathfrak{c}_{\lambda,m+k} }{	 f^{\lambda} }  \chi^{\lambda}([2, n-2]) \chi^{\lambda}(\beta) {\bigg |}_{\lambda=[1^{n-2},2]} &={1\choose{n-m-k}}(-1)^{\ell(\beta)} ,\\
		\sum_{m \geq 1} \sum_{k=0}^{n-m} (-1)^k \stirlingI{m+k}{m} \frac{n(n-1)}{(m+k)!}A_3 x^m
&=  n (n-1) {x+1 \choose n}(-1)^{\ell(\beta)}.
	\end{align}
\end{proposition}

\begin{proposition}\label{1,n-1}
	\begin{align}
		A_4= \frac{\mathfrak{c}_{\lambda,m+k} }{	 f^{\lambda} }  \chi^{\lambda}([2, n-2]) \chi^{\lambda}(\beta) {\bigg |}_{\lambda=[1,n-1]} &={{n-2}\choose{n-m-k}} , \\
		\sum_{m \geq 1} \sum_{k=0}^{n-m} (-1)^k \stirlingI{m+k}{m} \frac{n(n-1)}{(m+k)!}A_4 x^m
&= n (n-1) {x+n-2 \choose n}.
	\end{align}
\end{proposition}

It remains to deal with the cases $\lambda=[1^j,2^2, n-j-4]$ and $\lambda=[1^j, 3, n-j-3]$ where some
technical manipulation is required.
We will discuss the former in detail and the latter will follow from the same idea.
As usual, the notation $[y^n] f(y)$ means taking the coefficient of the term $y^n$ in the power series expansion of $f(y)$.

\begin{proposition}\label{2,n-2}
	The following equations hold:
	\begin{align}
		A_5= &\sum_{\lambda=[1^j,2^2,n-j-4], \atop 0\leq j \leq n-6} \frac{\mathfrak{c}_{\lambda,m+k} }{	 f^{\lambda} }  \chi^{\lambda}([2, n-2]) \chi^{\lambda}(\beta)\nonumber\\
=& \frac{(m+k)(m+k-1)}{n(n-1)} \Bigg\{ [y^{n-m-k}]\prod_{i=1}^d \big\{(1+y)^{\beta_{i}}-1\big\} \nonumber \\
 \qquad & -{n\choose{n-m-k}}-{0\choose{n-m-k}}(-1)^{\ell(\beta)} \Bigg\} ,
	\end{align}

\begin{align}
	& \sum_{m \geq 1}\sum_{k=0}^{n-m} (-1)^k \stirlingI{m+k}{m} \frac{n(n-1)}{(m+k)!}A_5 x^m
	=  [y^n] x(x-1) y^2 (1+y)^{x-2} \prod_i \Big\{ (1+y)^{\beta_i}-1 \Big\} \nonumber \\
	&\qquad \qquad \qquad  -n (n-1) {x+n-2 \choose n} - n(n-1){x\choose n}(-1)^{\ell(\beta)}.
\end{align}
\end{proposition}
\proof
According to Lemma~\ref{lem:2}, Lemma~\ref{lem:beta} and Lemma~\ref{lem:cf}, we first have
\begin{equation}
\begin{split}
& \quad \frac{n(n-1)}{(m+k)(m+k-1)} A_5
=\sum_{j=0}^{n-6}{{n-3-j}\choose{n-m-k}}(-1)^{j+1} \chi^{[1^j, 2^2,n-j-4]}(\beta )\\
&=\sum_{j=0}^{n-6}{{n-3-j}\choose{n-m-k}}\sum_{j_l}^{n_{l}-1}\sum_{j_{l+1}}^{n_{l}}\cdots\sum_{j_q}^{n_q}(-1)^{\sum_{i\geq l} j_{i}}{{n_{l}-1}\choose{j_{l}}}{n_{l+1}\choose j_{l+1}}\cdots{{n_q}\choose j_q}\\
&\qquad  \times \big\{\delta_{j+3,\sum_{i\geq l}ij_i}-\delta_{j+3-l,\sum_{i\geq l}ij_i} \big\}.
\end{split}\nonumber
\end{equation}	
Note that $0\leq \sum_{i \geq l} i j_i \leq n-l$, $3 \leq j+3 \leq n-3$, and
$3 - l \leq j+3 -l \leq n-3-l$.
Thus, for any possible combination of $j_i$ for $l \leq i \leq q$ except the case $j_i=0$ for all $i$,
there exists a unique $0\leq j \leq n-6$ such that $\delta_{j+3, \sum_i i j_i}=1$, i.e., $j=\sum_i i j_i-3$.
Similarly, for any possible combination of $j_i$ for $l \leq i \leq q$ except the case $j_i= n_i$ for all $l < i\leq q$
and $j_{l}= n_{l}-1$,
there exists a unique $0\leq j \leq n-6$ such that $\delta_{j+3-l, \sum_i i j_i}=1$, i.e., $j=\sum_i i j_i-3+ l$.
As a result, the last quantity equals
\begin{align*}
%	\begin{split}
		&\quad \sum_{j_{l+1}}^{n_{l}-1}\cdots\sum_{j_q}^{n_q}(-1)^{\sum_{i\geq l} j_{i}}{{n_{l}-1}\choose{j_{l}}}{n_{l+1}\choose j_{l+1}}\cdots{{n_q}\choose j_q}{{n-\sum_{i\geq l} ij_i}\choose{n-m-k}}-{n\choose{n-m-k}}\\
		& -\sum_{j_{l+1}}^{n_{l}-1}\cdots\sum_{j_q}^{n_q}(-1)^{\sum_{i\geq l} j_{i}}{{n_{l}-1}\choose{j_{l}}}{n_{l+1}\choose j_{l+1}}\cdots{{n_q}\choose q}{{n-\sum_{i\geq l} ij_i-l}\choose{n-m-k}} 
 -{0\choose{n-m-k}}(-1)^{\ell(\beta)}\\
		&=\sum_{j_{l+1}}^{n_{l}-1}\cdots\sum_{j_q}^{n_q}(-1)^{\sum_{i\geq l} j_{i}}{{n_{l}-1}\choose{j_{l}}}{n_{l+1}\choose j_{l+1}}\cdots{{n_q}\choose q}[y^{n-m-k}](1+y)^{n-\sum_{i\geq l}ij_i}\\
	&\qquad-\sum_{j_{l+1}}^{n_{l}-1}\cdots\sum_{j_q}^{n_q}(-1)^{\sum_{i\geq l} j_{i}}{{n_{l}-1}\choose{j_{l}}}{n_{l+1}\choose j_{l+1}}\cdots{{n_q}\choose q}[y^{n-m-k}](1+y)^{n-l-\sum_{i\geq l}ij_i}\\
		&\qquad  -{n\choose{n-m-k}}-{0\choose{n-m-k}}(-1)^{\ell(\beta)}.
\end{align*}	
It is the key to next realize that 
$$
\Big\{1-(1+y)^{-i}\Big\}^{n_i}= \sum_{j_i} {n_i \choose j_i} (-1)^{j_i} (1+y)^{-i j_i}.
$$
As such, the above formula containing the multiple summations can be simplified as
\begin{align*}
& [y^{n-m-k}]\Big\{1-(1+y)^{-l}\Big\}^{n_{l}-1}\Big\{1-(1+y)^{-l-1}\Big\}^{n_{l+1}}\cdots\Big\{1-(1+y)^{-q}\Big\}^{n_q}\\
& \qquad \qquad \times \Big\{(1+y)^{l}-1\Big\} (1+y)^{n-l} -{n\choose{n-m-k}}-{0\choose{n-m-k}}(-1)^{\ell(\beta) }.
\end{align*}
Note that $n-l = l (n_l-1) + \sum_{i=l+1}^q i n_i$. Consequently, the last formula is equal to
\begin{align*}
	& \quad [y^{n-m-k}]\Big\{(1+y)^{l}-1\Big\}^{n_{l}-1}\Big\{(1+y)^{l+1}-1\Big\}^{n_{l+1}}\cdots\Big\{(1+y)^{q} -1 \Big\}^{n_{q}} \\
	& \qquad \qquad \times \Big\{(1+y)^{l}-1\Big\}  -{n\choose{n-m-k}}-{0\choose{n-m-k}}(-1)^{\ell(\beta) }\\
&=[y^{n-m-k}]\prod_{i }\Big\{(1+y)^{\beta_{i}}-1\Big\}-{n\choose{n-m-k}}-{0\choose{n-m-k}}(-1)^{\ell(\beta) } .
\end{align*}

Next, we compute the generating function of $A_5$:
	\begin{align*}
	& \quad	\sum_{m \geq 1} \sum_{k=0}^{n-m} (-1)^k \stirlingI{m+k}{m} \frac{n(n-1)}{(m+k)!}A_5 x^m\\
		&=\sum_{m \geq 1} \sum_{k=0}^{n-m} (-1)^k \stirlingI{m+k}{m} \frac{(m+k)(m+k-1)}{(m+k)!} [y^{n-m-k}]\prod_{i}\Big\{(1+y)^{\beta_{i}}-1\Big\} x^m\\
& \quad -\sum_{m \geq 1} \sum_{k=0}^{n-m} (-1)^k \stirlingI{m+k}{m} \frac{(m+k)(m+k-1)}{(m+k)!} {n\choose{n-m-k}}x^m\\
& \quad -\sum_{m \geq 1} \sum_{k=0}^{n-m} (-1)^k \stirlingI{m+k}{m} \frac{(m+k)(m+k-1)}{(m+k)!} {0\choose{n-m-k}} (-1)^{\ell(\beta) }x^m\nonumber\\
&=\sum_{m \geq 1} \sum_{k\geq1} (-1)^{k-m} \stirlingI{k}{m} \frac{k(k-1)}{k!} [y^{n-k}]\prod_{i}\Big\{(1+y)^{\beta_{i}}-1\Big\} x^m\nonumber\\
&\quad-\sum_{m \geq 1} \sum_{k\geq1} (-1)^{k-m} \stirlingI{k}{m} \frac{k(k-1)}{k!} {n\choose{n-k}}x^m\nonumber\\
&\quad-\sum_{m \geq 1} \sum_{k\geq1} (-1)^{k-m} \stirlingI{k}{m} \frac{k(k-1)}{k!} {0\choose{n-k}}(-1)^{\ell(\beta) } x^m\nonumber\\
&=\sum_{k\geq1}{x\choose k}k(k-1)[y^{n-k}]\prod_{i}\Big\{(1+y)^{\beta_{i}}-1\Big\}\nonumber\\
&\quad-\sum_{k\geq1}{x\choose k}k(k-1){n\choose{n-k}}-{x\choose n}n(n-1)(-1)^{\ell(\beta)} \\
&= [y^n] x(x-1) y^2 (1+y)^{x-2} \prod_i\Big \{ (1+y)^{\beta_i}-1 \Big\} \\
&\quad -n(n-1) {x+n-2 \choose n} - n(n-1){x\choose n}(-1)^{\ell(\beta) }.
	\end{align*}
This completes the proof. \qed

Analogous to the last proposition, we obtain

\begin{proposition}\label{1,3}
	\begin{align}
		A_6&= \sum_{\lambda=[1^j,3,n-j-3], \atop 0\leq j \leq n-6} \frac{\mathfrak{c}_{\lambda,m+k} }{	 f^{\lambda} }  \chi^{\lambda}([2, n-2]) \chi^{\lambda}(\beta) \nonumber\\
&= [y^{n-m-k}](1+y)\prod_{i\geq l}\Big\{(1+y)^{\beta_{i}}-1\Big\}-{{n+1}\choose{n-m-k}}-{1\choose{n-m-k}}(-1)^{\ell(\beta) } \nonumber\\
& \quad + [y^{n-m-k-1}]\prod_{i\geq l}\Big\{(1+y)^{\beta_{i}}-1\Big\}-{n\choose{n-m-k-1}}-{0\choose{n-m-k-1}}(-1)^{\ell(\beta) }, 
	\end{align}

\begin{align*}
	& \sum_{m \geq 1} \sum_{k=0}^{n-m} (-1)^k \stirlingI{m+k}{m} \frac{n(n-1)}{(m+k)!}A_6 x^m\\
=& [y^n]  x(x+1) y^2 (1+y)^{x-1} \prod_i [(1+y)^{\beta_i}-1] - x(x-1){x+n-1 \choose n-2} \\
	& \qquad -x(x-1) {x-1 \choose n-2} (-1)^{l(\beta)}
	 -2 x {x+n-1 \choose n-2} - 2x {x-1 \choose n-2} (-1)^{\ell(\beta)}.
\end{align*}

\end{proposition}

Finally, we derive the generating polynomial analogous to that of Stanley mentioned earlier.

 \begin{theorem}\label{thm:main-poly}
 	 Let $\beta=(\beta_1, \beta_2, \ldots, \beta_d) \vdash n$ where $\beta_i \geq 3$ and $d\geq 1$. Then,
 	we have
 	\begin{align}\label{eq:main-poly}
 		P_{[2,n-2], \beta} (x) &= [y^n] \, x y^2 (1+y)^{x-2}  ( 2x+xy+y) \prod_{i=1}^d \Big\{(1+y)^{\beta_{i}}-1\Big\} .
 	\end{align}
 \end{theorem}
\proof 
First, it is not difficult to see
$$
P_{[2,n-2], \beta} (x) =\sum_{m\geq 1}\frac{n(n-1)2(n-2)\xi_{n,m}([2,n-2],\beta)}{n!|\mathcal{C}_{\beta}|}x^m.
$$
Summarizing Proposition~\ref{n} to Proposition~\ref{1,3}, we next have
\begin{align*}
P_{[2,n-2], \beta} (x)
	&=n(n-1) {x+n-1 \choose n} + n(n-1){x\choose n}(-1)^{\ell(\beta)} + n (n-1) {x+1 \choose n}(-1)^{\ell(\beta)} \\
	& \quad + n (n-1) {x+n-2 \choose n} + [y^n] x(x-1) y^2 (1+y)^{x-2} \prod_i \{ (1+y)^{\beta_i}-1 \} \\
	& \quad -n(n-1) {x+n-2 \choose n}
	 -n(n-1){x\choose n}(-1)^{\ell(\beta)}\\
	& \quad + [y^n]  x(x+1) y^2 (1+y)^{x-1} \prod_i [(1+y)^{\beta_i}-1] - x(x-1){x+n-1 \choose n-2} \\
	& \quad -x(x-1) {x-1 \choose n-2} (-1)^{\ell(\beta)}
	-2 x {x+n-1 \choose n-2} - 2x {x-1 \choose n-2} (-1)^{\ell(\beta)} .
\end{align*}
The last formula can be easily simplified to
\begin{align*}
	[y^n] x y^2 (1+y)^{x-2}  ( 2x+xy+y) \prod_{i}\Big\{(1+y)^{\beta_{i}}-1\Big\},
\end{align*}
and the proof follows. \qed

Here are some examples of $P_{[2,n-2], \beta}(x)$:
\begin{align*}
	P_{[2,4], [3^2]}(x) &= 21 x^2 + 9 x^4,\\
	P_{[2,5], [3,4]}(x) &= 10x + 28 x^3 +4 x^5,\\
	P_{[2,6], [4^2]}(x) &= \frac{4}{3} x (25 x + 16 x^3 + x^5),\\
	P_{[2,7], [3^3]}(x) &= \frac{9}{4} x (18x +13 x^3 + x^5).
\end{align*}

The following corollary may be viewed as an analogue
of the Harer-Zagier formula~\cite{harer-zagier}, both dealing with regular edge-types.

\begin{corollary} For $p\geq 3$, we have
\begin{align}
	P_{[2,np-2],[p^n]} (x) &= \sum_{i=0}^n (-1)^{n-i} {n \choose i} \Bigg[ 2 x^2 {x-2+pi \choose np-2} + x (x+1)   {x-2+pi \choose np-3} \Bigg].
\end{align}
\end{corollary}
\proof According to Theorem~\ref{thm:main-poly}, the computation is straightforward as below:
\begin{align*}
	P_{[2,np-2], [p^n]} (x) &= [y^{np}] x y^2 (1+y)^{x-2}  ( 2x+xy+y) \Big\{(1+y)^{p}-1\Big\}^n \\
	&=[y^{np-2}] x  (1+y)^{x-2}  ( 2x+xy+y) \sum_{i=0}^n {n \choose i} (1+y)^{pi} (-1)^{n-i}\\
	&=[y^{np-2}] 2x^2 \sum_{i=0}^n {n \choose i} (1+y)^{pi+x-2} (-1)^{n-i}\\
	& \quad + [y^{np-3}] x (x+1) \sum_{i=0}^n {n \choose i} (1+y)^{pi+x-2} (-1)^{n-i}\\
	&= 2 x^2 \sum_{i=0}^n (-1)^{n-i} {n \choose i} {x-2+pi \choose np-2} + x (x+1) \sum_{i=0}^n (-1)^{n-i} {n \choose i} {x-2+pi \choose np-3}.
\end{align*}
\qed

Thanks to the formula in Theorem~\ref{thm:main-poly}, subsequently we are able to express
$P_{[2,n-2],\beta}(x)$ as a linear combination of $n$ simple polynomials.
First, for $\beta=(\beta_1, \ldots, \beta_d) \vdash n$, let
\begin{align}
	\rho(\beta) = \prod_{i=1}^{d}( x^{\beta_i}-y^{\beta_i}).
\end{align}
Let $\theta_{i,n} =[1^{n-i}, i] \vdash n$. We simply write $\theta_{i,n}$ as $\theta_i$ when
$n$ is clear from the context.

\begin{theorem}[Chen and Wang~\cite{chw}] \label{thm:chw}
	
	For any $\beta \vdash n$, we have $(a_{n,\beta},a_{n-1, \beta},\ldots,a_{1,\beta})$ such that	
	$$
	\rho(\beta) = a_{n,\beta}\rho(\theta_{n})+a_{n-1,\beta}\rho(\theta_{n-1})+\cdots+ a_{1,\beta}\rho(\theta_{1}),
	$$
	and the coefficients satisfy:
		$$
	\begin{bmatrix}
		a_{n,\beta} \\
		a_{n-1,\beta}\\
		\vdots \\
		a_{1,\beta}
	\end{bmatrix}
	=
	\begin{bmatrix}
		\rho^{1}_{\theta_{n}} & \ & \ & \ \\
		\rho^{2}_{\theta_{n}} & \rho^{2}_{\theta_{n-1}} & \ & \ \\
		\vdots &  \vdots & \ddots & \ \\
		\rho^{n}_{\theta_{n}} & \rho^{n}_{\theta_{n-1}} & \cdots &  \rho^{n}_{\theta_{1}}
	\end{bmatrix}^{-1}
	\times
	\begin{bmatrix}
		\rho^{1}_{\beta} \\
		\rho^{2}_{\beta} \\
		\vdots \\
		\rho^{n}_{\beta}
	\end{bmatrix},
	$$
	where for $\gamma=(\gamma_1, \gamma_2, \ldots, \gamma_t)$,
	\begin{align}
		\rho^k_{\gamma}=&
		\sum_{b_1+\cdots+b_t= k, \atop b_i\geq 1 }\binom{k}{b_1,\ldots,b_t}(\gamma_1)_{b_1} (\gamma_2)_{b_2}\ldots(\gamma_t)_{b_t} .
	\end{align}
\end{theorem}

Now we have the following decomposition theorem.

\begin{theorem}
	Let $\beta= (\beta_1, \beta_2,\ldots, \beta_d) \vdash n$ where $d\geq 1$ and $\beta_i \geq 3$.
	Then, we have
	\begin{align}
		P_{[2,n-2], \beta}(x)  = a_{n,\beta} \tilde{P}_{\theta_{n}}(x )+a_{n-1,\beta} \tilde{P}_{\theta_{n-1}}(x )+\cdots+ a_{1,\beta} \tilde{P}_{\theta_{1}}(x ),
	\end{align}
	where
	\begin{align}
		\tilde{P}_{\theta_{i}}(x ) 
		&= 2 x^2 \left[  {i+x-2 \choose i-2}- {x-2 \choose i-2} \right] + x(x+1) \left[ {i+x-2 \choose i-3}- {x-2 \choose i-3}\right].
	\end{align}
\end{theorem}

\proof
According to Theorem~\ref{thm:chw}, we have
\begin{align*}
\prod_i \big\{ (1+y)^{\beta_i} -1 \big\} &= \rho({\beta}) \big|_{x=1+y, \, y=1}\\
&= \sum_{i} a_{i, \beta} \rho(\theta_i) \big|_{x=1+y, \, y=1} = \sum_{i} a_{i, \beta}  y^{n-i} \Big\{(1+y)^{i}-1 \Big\}.
\end{align*}
Plugging the last relation into eq.~\eqref{eq:main-poly}, we obtain
\begin{align*}
	P_{[2,n-2], \beta}(x)&= \sum_i a_{i, \beta} [y^n] x y^2 (1+y)^{x-2}  ( 2x+xy+y) y^{n-i} \Big\{(1+y)^{i}-1 \Big\}\\
	&= \sum_i a_{i, \beta}  \left\{ 2 x^2 \Big[ {i+x-2 \choose i-2}- {x-2 \choose i-2} \Big] + x(x+1) \Big[ {i+x-2 \choose i-3}- {x-2 \choose i-3}\Big] \right\},
\end{align*}
completing the proof.

Note that we used $\tilde{P}_{\theta_{i}}(x )$ in order to distinguish it from ${P}_{\theta_{i}}(x )$
for which we have not presented an explicit formula yet.

\subsection{$\beta=[2, n-2]$ and $\beta=[1, n-1]$}

The computation for $\beta=[2, n-2]$ and $\beta=[1, n-1]$
is similar but easier than the general $\beta$ in the last section.
We leave the details to the interested reader and only present the results.

\begin{theorem}
	 We have
	\begin{align}
&P_{[2,n-2], [2,n-2]}(x)=	 x(x+1)\left\{{x+n-1 \choose{n-2}}+{{x-1}\choose{n-2}}+{x\choose{n-2}}+ {x+n-3 \choose{n-1}}-{{x+2}\choose{n-1}}\right\} \nonumber\\
	& \quad +  
x(x-1)\left\{{{x-2}\choose{n-2}}+{x+n-2 \choose{n-2}}+{x+n-3 \choose{n-2}}+ {x+n-4 \choose{n-1}}-{{x+1}\choose{n-1}}\right\}.
\end{align}
\end{theorem}

Here are some examples for this case:
\begin{align*}
	P_{[2,3],[2,3]}(x) &= 6 x + 13 x^3 + x^5,\\
	P_{[2,4],[2,4]}(x) &= 20 x^2 + \frac{29}{3} x^4 + \frac{1}{3} x^6,\\
	P_{[2,5],[2,5]}(x) &= 10 x + \frac{82}{3} x^3 + \frac{55}{12} x^5 + \frac{1}{12} x^7.
\end{align*}

We remark that plugging $\beta=[2, n-2]$ into the formula eq.~\eqref{eq:main-poly} will not give the correct result.
For example, for $\beta=[2,3]$, eq.~\eqref{eq:main-poly} returns $6x+12 x^3$; and for $\beta=[2,4]$, eq.~\eqref{eq:main-poly} returns $20 x^2 + 8 x^4$.

\begin{theorem}
	\begin{align}
		\frac{P_{[2,n-2], [1,n-1]}(x)}{n-1} = x{ x+n -1 \choose{n-1}}+x{{x-1}\choose{n-1}} -x{{x+1}\choose{n-1}} - x{{x+n-3}\choose {n-1}}.
	\end{align}
\end{theorem}

\section{Zeros and log-concavity}\label{sec4}

In this section, we discuss the zeros and the coefficients of the obtained polynomials in the last section.
First of all, by a parity argument of permutations, it is well known that $P_{[2,n-2], \beta}(x)$
either has only terms of odd degrees or only even degrees. This is not immediately clear from the expression
in Theorem~\ref{thm:main-poly}, and we will illustrate this through the following proposition.

\begin{proposition}\label{prop:parity}
	Suppose $\beta=(\beta_1, \beta_2, \ldots, \beta_d) \vdash n$ and $\beta_i \geq 3$. Then,
	we have
	\begin{align}
		P_{[2,n-2], \beta} (x) &= x(x-1) G(x) + x(x+1) (-1)^{n-d} G(-x) ,
	\end{align}
	where
	$$
	G(x)= [y^n] \, y^2 (1+y)^{x-2}  \prod_{i=1}^d \Big\{(1+y)^{\beta_{i}}-1\Big\}.
	$$
\end{proposition}

\proof Following the proof of Theorem~\ref{thm:main-poly}, we first have
\begin{align*}
	P_{[2,n-2], \beta}(x) &= [y^n ] x(x-1)   y^2 (1+y)^{x-2}  \prod_{i}\Big\{(1+y)^{\beta_{i}}-1\Big\} \\
	& \qquad +
	[y^n]  x(x+1) y^2 (1+y)^{x-1} \prod_{i}\Big\{(1+y)^{\beta_{i}}-1\Big\}.
\end{align*}
In the previous computations, we also saw that
\begin{align*}
	&\quad [y^n ] x(x-1)   y^2 (1+y)^{x-2}  \prod_{i=1}^d\Big\{(1+y)^{\beta_{i}}-1\Big\} \\
	&=\sum_{k\geq1}{x\choose k}k(k-1)[y^{n-k}]\prod_{i=1}^d\Big\{(1+y)^{\beta_{i}}-1\Big\}\\
	& =\sum_{k\geq1}{x\choose k}k(k-1)[y^{n-k}] \Big\{(1+y)^{{n-\beta_1-\cdots -\beta_{d-1}}}-1\Big\}  \prod_{i=1}^{d-1}\Big\{(1+y)^{\beta_{i}}-1\Big\} \\
	&= [z^n] \sum_{n} \sum_{k\geq1}{x\choose k}k(k-1)[y^{n-k}] \Big\{(1+y)^{{n-\beta_1-\cdots -\beta_{d-1}}}\Big\} \prod_{i=1}^{d-1}\Big\{(1+y)^{\beta_{i}}-1\Big\} z^n\\
	&\quad -[z^n] \sum_{n} \sum_{k\geq1}{x\choose k}k(k-1)[y^{n-k}] \prod_{i=1}^{d-1}\Big\{(1+y)^{\beta_{i}}-1\Big\} z^n .
\end{align*}
The last line is easily seen to be
\begin{align}
	&[y^n] \, -x(x-1) y^2 (1+y)^{x-2} \prod_{i=1}^{d-1}\Big\{(1+y)^{\beta_{i}}-1\Big\} \nonumber \\
	=&x(x-1) \sum_{1\leq j_1<\cdots < j_i \leq d-1 \atop  i\geq 0}  (-1)^{d-i} {x-2+\beta_{j_1}+\cdots +\beta_{j_i} \choose n-2}. \label{eq:zero1}
\end{align}
We proceed to compute the second last line. It equals
\begin{align}
	& [z^n] \sum_{n} \sum_{k\geq1}{x\choose k}k(k-1) \sum_{1\leq j_1<\cdots < j_i \leq d-1 \atop  i\geq 0}  {n-\beta_{j_1}-\cdots -\beta_{j_i} \choose n-k } (-1)^{i} z^n  \nonumber\\
	= &[z^n]  \sum_{k\geq1}{x\choose k}k(k-1) \sum_{1\leq j_1<\cdots < j_i \leq d-1 \atop  i\geq 0} (-1)^{i} \sum_{n} {n-\beta_{j_1}-\cdots -\beta_{j_i} \choose n-k }  z^n \nonumber \\
	= &  [z^n ]  \sum_{k\geq1}{x\choose k}k(k-1) \sum_{1\leq j_1<\cdots < j_i \leq d-1 \atop  i\geq 0} (-1)^{i} \frac{z^{k}}{(1-z)^{k -\beta_{j_1}-\cdots - \beta_{j_i}+1}} \nonumber \\
	=& [y^n] x(x-1) y^2 (1-y)^{-x-1} \prod_{i=1}^{d-1} \big\{ -(1-y)^{\beta_i} +1\big\} \nonumber\\
	=& x(x-1) \sum_{1\leq j_1<\cdots < j_i \leq d-1 \atop  i\geq 0} (-1)^i {-x-1+\beta_{j_1}+\cdots +\beta_{j_i} \choose n-2} (-1)^{n-2}. \label{eq:zero2}
\end{align}
Similarly, we have
\begin{align}
	& \quad [y^n]  x(x+1) y^2 (1+y)^{x-1} \prod_{i=1}^d \Big\{(1+y)^{\beta_{i}}-1\Big\} \nonumber\\
	&=	[y^n] x(x+1) y^2 (1-y)^{-x-2} \prod_{i=1}^{d-1} \big\{ -(1-y)^{\beta_i} +1\big\} \nonumber\\
	& \quad -	[y^n] x(x+1) y^2 (1+y)^{x-1} \prod_{i=1}^{d-1} \Big\{(1+y)^{\beta_{i}}-1\Big\} \nonumber\\
	&= x(x+1) \sum_{1\leq j_1<\cdots < j_i \leq d-1 \atop  i\geq 0}  (-1)^{n-2+i} {-x-2 +\beta_{j_1}+\cdots + \beta_{j_i} \choose n-2} \label{eq:zero3}\\
	& \quad + x(x+1) \sum_{1\leq j_1<\cdots < j_i \leq d-1 \atop  i\geq 0} (-1)^{d-i} {x-1 +\beta_{j_1}+\cdots + \beta_{j_i} \choose n-2}. \label{eq:zero4}
\end{align}
Comparing eq.~\eqref{eq:zero1}, eq.~\eqref{eq:zero2} and eq.~\eqref{eq:zero3}, eq.~\eqref{eq:zero4},
we easily arrive at the proposition.  \qed

Next, the following result of Stanley~\cite{Stanely-two} will be useful to us.
\begin{theorem}[Stanley~\cite{Stanely-two}]\label{thm:stan-zero}
	Let $\textbf{E}$ be the operator $\textbf{E}: f(x) \rightarrow f(x-1)$.
	Suppose $g(t)$ is a complex polynomial with all roots on the unit circle,
	and the multiplicity of $1$ as a root is $m\geq 0$.
	Let $Q(x)= g(\textbf{E}) (x+n-1)_n$.
	If $g(t)$ has degree exactly $d \geq n-1$, then $Q(x)$ is a polynomial of degree $n-m$ for which
	every zero has real part $\frac{d-n+1}{2}$.
\end{theorem}

Now we are ready to prove the imaginary zero property below.
\begin{theorem}[Imaginary zeros]
	If $\beta=(\beta_1, \beta_2, \ldots, \beta_d) \vdash n$ where $\beta_i \geq 3$, or $\beta=[1,n-1]$,
	then every zero of the polynomial $P_{[2,n-2], \, \beta}(x)$
	has real part zero.
\end{theorem}
\proof 
For the former case, we first have
\begin{align*}
	G(x)&=[y^n] \, y^2 (1+y)^{x-2}  \prod_{i=1}^d \Big\{(1+y)^{\beta_{i}}-1\Big\}  \\
	&=\sum_{1\leq j_1<\cdots < j_i \leq d \atop i \geq 0} {x-2+\beta_{j_1}+\cdots + \beta_{j_i} \choose n-2} (-1)^{d-i}   \\
	&=\sum_{1\leq j_1<\cdots < j_i \leq d \atop i \geq 0} {x-2+n-\beta_{j_1}-\cdots - \beta_{j_i} \choose n-2} (-1)^{i}.
\end{align*}
Denote by $\widetilde{G}(y)$ the polynomial resulted from replacing $x+1$ with $y$ in $G(x)$.
Then, in terms of the backward shift operator $\textbf{E}$, it is not difficult to see
$$
\widetilde{G}(y) = \prod_{i=1}^d (-\textbf{E}^{\beta_i} +1) {y+n-3 \choose n-2}.
$$
According to Theorem~\ref{thm:stan-zero}, we conclude that every zero of $\widetilde{G}(y)$
has real part $\frac{n-(n-3)}{2}=\frac{3}{2}$.
Consequently, all zeros of $G(x)$ have real part $\frac{1}{2}$.

Let $h(x)=x(x-1)G(x)$. We then observe that
$$
P_{[2,n-2], \beta}(x)= h(x)+  h(x+1).
$$ 
For a complex number $z$,
if $P_{[2,n-2], \beta} (z)=0$, then we have the following modulus relation
$$
|h(z)| = |h(z+1)|.
$$
Since every zero of $G(x)$ has real part $1/2$,
we may assume
$$
h(x)=x (x-1) \prod_j (x-\frac{1}{2}-i b_j),
$$
where $i^2=-1$. We also assume $z=p+iq$. Then, the above modulus relation becomes
$$
|p+iq|\cdot |p-1+iq| \cdot \prod_j |p-\frac{1}{2} +i (q-b_j)|=|p+iq|\cdot |p+1+iq| \cdot \prod_j |p+\frac{1}{2}  +i (q-b_j)|.
$$
If $p>0$,
we have $|p+iq|\neq 0$, $|p-1+iq|< |p+1+iq|$ and $|p-\frac{1}{2} +i (q-b_j)| < |p+\frac{1}{2}  +i (q-b_j)|$ for all $j$.
That is, the lefthand side is strictly smaller than the righthand side.
If $p<0$, we have strict inequality in the opposite direction. Therefore, $p=0$ is necessary,
and the former case follows.

In terms of the backward shift operator $\textbf{E}$, the latter case leads to
$$
\frac{P_{[2,n-2], [1,n-1]}}{(n-1)x} = (1+ \textbf{E}^n -\textbf{E}^{n-2}- \textbf{E}^2)((x+1)+n-2)_{n-1}.
$$
The rest is clear from 
Theorem~\ref{thm:stan-zero}.
This completes the proof. \qed

\begin{corollary}[Log-concavity and unimodality]
	If $\beta=(\beta_1, \beta_2, \ldots, \beta_d) \vdash n$ where $\beta_i \geq 3$, or $\beta=[1,n-1]$,
	then both the sequences $\xi_{n,1}([2,n-2], \beta), \xi_{n,3}([2,n-2], \beta),\xi_{n,5}([2,n-2], \beta),\ldots$
	and $\xi_{n,2}([2,n-2], \beta), \xi_{n,4}([2,n-2], \beta),\xi_{n,6}([2,n-2], \beta),\ldots$ are log-concave and unimodal.
\end{corollary}
\proof In view of Proposition~\ref{prop:parity}, the polynomial
$$
P_{[2,n-2], \beta}(x) =\sum_{m\geq 1} \frac{n(n-1)}{|\mathcal{C}_{[2,n-2]}| \cdot |\mathcal{C}_{\beta}|} \xi_{n,m} ([2,n-2],\beta) x^m
$$
 has either only odd-degree terms or only even-degree terms.
 Suppose only even-degree terms exist. Then,
 $$
 P_{[2,n-2], \beta}(x) =\frac{n(n-1)}{|\mathcal{C}_{[2,n-2]}| \cdot |\mathcal{C}_{\beta}|} \big\{ \xi_{n,2} ([2,n-2],\beta) x^2 + \xi_{n,4} ([2,n-2],\beta) x^4 + \cdots \big\} .
 $$
 The polynomial $P_{[2,n-2], \beta}(x) $ having only imaginary zeros implies that the polynomial 
 $$
  \xi_{n,2} ([2,n-2],\beta) y + \xi_{n,4} ([2,n-2],\beta) y^2 + \cdots 
 $$
has only real zeros. As is well known, the latter implies the log-concavity of the corresponding coefficients and thus unimodality.
In this case, $\xi_{n,1}([2,n-2], \beta), \xi_{n,3}([2,n-2], \beta),\xi_{n,5}([2,n-2], \beta),\ldots$ is just a sequence of zeros which
is clearly log-concave and unimodal.
The situation that $P_{[2,n-2], \beta}(x)$ has only odd-degree terms is analogous, completing the proof. \qed

\end{document}